\documentclass[12pt,reqno]{article}

\usepackage[usenames]{color}
\usepackage{amssymb}
\usepackage{amsmath}
\usepackage{amsthm}
\usepackage{amsfonts}
\usepackage{amscd}
\usepackage{graphicx}

\usepackage[colorlinks=true,
linkcolor=webgreen,
filecolor=webbrown,
citecolor=webgreen]{hyperref}

\definecolor{webgreen}{rgb}{0,.5,0}
\definecolor{webbrown}{rgb}{.6,0,0}

\usepackage{color}
\usepackage{fullpage}
\usepackage{float}

\usepackage{graphics}
\usepackage{latexsym}
\usepackage{epsf}
\usepackage{breakurl}

\def\suchthat{\, : \,}
\def\Enn{\mathbb{N}}

\newcommand{\seqnum}[1]{\href{https://oeis.org/#1}{\rm \underline{#1}}}

\begin{document}

\theoremstyle{plain}
\newtheorem{theorem}{Theorem}
\newtheorem{corollary}[theorem]{Corollary}
\newtheorem{lemma}[theorem]{Lemma}
\newtheorem{proposition}[theorem]{Proposition}

\theoremstyle{definition}
\newtheorem{definition}[theorem]{Definition}
\newtheorem{example}[theorem]{Example}
\newtheorem{conjecture}[theorem]{Conjecture}

\theoremstyle{remark}
\newtheorem{remark}[theorem]{Remark}

\title{Some Tribonacci Conjectures}

\author{Jeffrey Shallit\\
School of Computer Science \\
University of Waterloo \\
Waterloo, ON  N2L 3G1 \\
Canada\\
\href{mailto:shallit@uwaterloo.ca}{\tt shallit@uwaterloo.ca}}

\maketitle

\begin{abstract}
In a recent talk of Robbert Fokkink, some conjectures related
to the infinite Tribonacci word were stated by the speaker and the audience.
In this note we show how to prove (or disprove) the claims
easily in a ``purely mechanical'' fashion, using the
{\tt Walnut} theorem-prover.   
\end{abstract}

\section{Introduction}

The {\it Tribonacci sequence} $(T_n)_{n \geq 0}$ satisfies the
three-term recurrence relation
$$ T_n = T_{n-1} + T_{n-2} + T_{n-3}$$
for $n \geq 3$, with initial values
$T_0 = 0$, $T_1 = 1$, $T_2 = 1$; see
\cite{Feinberg:1963,Scott&Delaney&Hoggatt:1977}.
It is sequence \seqnum{A000073} in the 
{\it On-Line Encyclopedia of Integer Sequences} (OEIS); see
\cite{Sloane:2022}.  The first few values are as follows:
\begin{table}[H]
\begin{center}
\begin{tabular}{c|ccccccccccccccc}
$n$ &  1& 2& 3& 4& 5& 6& 7& 8& 9&10&11&12&13&14&15\\
\hline
$T_n$ & 1 & 1&   2&   4&   7&  13&  24&  44&  81& 149& 274& 504& 927&1705&3136\\
\end{tabular}
\end{center}
\caption{First few values of the Tribonacci sequence.}
\end{table}

As is well-known, every non-negative integer $N$ has a
unique representation in the form
$$N = \sum_{2 \leq i \leq t} e_i T_i,$$
where $e_t = 1$, $e_i \in \{0,1\}$, and $e_i e_{i+1} e_{i+2} \not= 1$ for
$i \geq 2$ \cite{Carlitz&Scoville&Hoggatt:1972d}.
This representation is conventionally written as a
binary word (or string) $(N)_T$ with the most significant digit first:
$e_t e_{t-1} \cdots e_2$.   For example, $(43)_T = 110110$.
There is a related function
$[x]_T = N$, if $x = x_1 \cdots x_t \in \{0,1\}^*$ and
$N = \sum_{1 \leq i \leq t} x_i T_{t+3-i}$.

Intimately connected with the Tribonacci sequence is the 
{\it infinite Tribonacci word} 
$$ {\bf TR}  = t_0 t_1 t_2 \cdots = 0102010 \cdots ,$$
where 
$t_N$ is defined to be the number of trailing $1$'s in
$(N)_T$.  This is a famous infinite word that has been
extensively studied (e.g.,
\cite{Barcucci&Belanger&Brlek:2004,Tan&Wen:2007,Richomme&Saari&Zamboni:2010,Huang&Wen:2017,Lejeune&Rigo&Rosenfeld:2020}).  It is sequence
\seqnum{A080843} in the OEIS.

The Tribonacci sequence and its relatives
are intimately connected with a number
of combinatorial games \cite{Duchene&Rigo:2008}.
In this note I will
explain how one can rigorously prove properties of these sequences
``purely mechanically'', using various computational techniques involving
formal languages and automata.
As an example of the method, we show how to prove (or disprove!) some
recent conjectures made in a talk by Robbert Fokkink.
One of these conjectures also appears just before Section 2 in
\cite{Fokkink&Rust:2022}.

We use the {\tt Walnut} theorem-prover
originally designed by Hamoon Mousavi \cite{Mousavi:2016}.
This free software is able to rigorously prove or disprove first-order
statements involving automatic and synchronized sequences.
For previous work with {\tt Walnut} and the infinite Tribonacci
word, see \cite{Mousavi&Shallit:2015}.

\section{Automata}
\label{aut-sec}

Our methods involve two classes of infinite words, 
called Tribonacci-automatic and Tribonacci-synchronized.
For more about them, see \cite{Mousavi&Shallit:2015,Shallit:2022}.

An infinite word ${\bf x}$ is said to be {\it Tribonacci-automatic\/} if
there is a deterministic finite automaton with output (DFAO) that,
on input $[N]_T$, the Tribonacci representation of $N$, reaches
a state with output ${\bf x}[N]$.  For more about automatic sequences,
see \cite{Allouche&Shallit:2003}.

An infinite sequence of integers $(y_n)_{n \geq 0}$ is said to
be {\it Tribonacci-synchronized\/} if there is a deterministic finite
automaton (DFA) that, on input the Tribonacci representations of
$N$ and $x$ in parallel (the shorter one, if necessary, padded
with leading zeros), accepts if and only if $x = y_N$.
For more about synchronized sequences, see
\cite{Shallit:2021h}.

The Tribonacci word $\bf TR$ is Tribonacci-automatic, and is
generated by the automaton in Figure~\ref{fig1}, where the output of the
state numbered $i$ is $i$ itself.
\begin{figure}[htb]
\begin{center}
\includegraphics[width=4in]{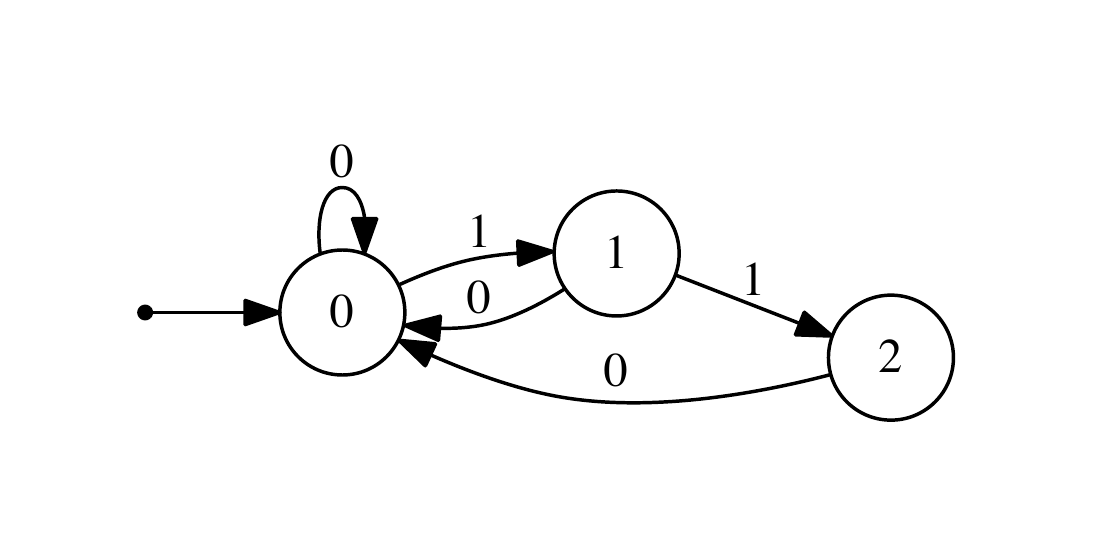}
\end{center}
\caption{Tribonacci automaton computing $\bf TR$.}
\label{fig1}
\end{figure}

Three examples of Tribonacci-synchronized sequences are the
sequences $(D_n)_{n \geq 0}$, $(E_n)_{n \geq 0}$,
$(F_n)_{n \geq 0}$ defined as follows:
\begin{align*}
D_n &= |{\bf TR}[0..n-1]|_0 \\
E_n &= |{\bf TR}[0..n-1]|_1 \\
F_n &= |{\bf TR}[0..n-1]|_2 .
\end{align*}
Tribonacci automata for these sequences are given in \cite[\S 10.12]{Shallit:2022}.  These are sequences \seqnum{A276796}, \seqnum{A276797}, and 
\seqnum{A276798}$-1$ in the OEIS, respectively.

\section{Two sequences related to combinatorial games}

Our starting point is two sequences related to combinatorial games.
These are sequences \seqnum{A140100} and \seqnum{A140101} in the
OEIS, from which we quote the definition more-or-less verbatim:
\begin{definition}
Start with $X(0) = 0$, $Y(0)=0$, $X(1)=1$, $Y(1)=2$.
For $n > 1$, choose the least positive integers $Y(n) > X(n)$
such that neither $Y(n)$ nor $X(n)$ appear in 
$\{Y(k) \suchthat  1 \leq k < n \}$ or
$\{X(k) \suchthat 1 \leq k < n \}$ and such that $Y(n)-X(n)$
does not appear in $\{Y(k)-X(k) \suchthat 1 \leq k < n\}$
or $\{Y(k)+X(k) \suchthat 1 \leq k < n \}$.
\label{def}
\end{definition}
Table~\ref{tab2} gives the first few terms of these sequences.
\begin{table}[H]
\begin{center}
\begin{tabular}{c|cccccccccccccccc}
$n$ & 0 &  1& 2& 3& 4& 5& 6& 7& 8& 9&10&11&12&13&14&15\\
\hline
$X(n)$ & 0 & 1& 3& 4& 6& 7& 9&10&12&14&15&17&18&20&21&23\\
$Y(n)$ & 0 & 2& 5& 8&11&13&16&19&22&25&28&31&33&36&39&42
\end{tabular}
\end{center}
\caption{First few values of the sequences $(X(k))_{k \geq 0}$ and
$(Y(k))_{k \geq 0}$.}
\label{tab2}
\end{table}
These two sequences were apparently first studied by
Duch\^{e}ne and Rigo \cite[Corollary 3.6]{Duchene&Rigo:2008} in connection with
a combinatorial game, and they established the close connection between
them and the infinite Tribonacci word $\bf TR$.   Further properties
were given in \cite{Dekking&Shallit&Sloane:2020}.

Because of the already-established connection with $\bf TR$, one might
conjecture that the sequences ${\bf X} = (X(k))_{k \geq 0}$ and 
${\bf Y} = (Y(k))_{k \geq 0}$ are both
Tribonacci-synchronized.   
But how can this be verified?
There is no known method to take
an arbitrary recursive description of sequences, like the one given above in
Definition~\ref{def}, and turn it into automata computing the sequences.

We do so in two steps.  The first is to use some heuristics
to ``guess'' the automata for
$\bf X$ and $\bf Y$.   We do that using (essentially) the classical
Myhill-Nerode theorem, as described below for an arbitrary sequence
of natural numbers.

\subsection{Guessing an automaton}

Let $(x_n)_{n \geq 0}$ be a sequence of natural numbers that we suspect
is Tribonacci-synchronized.  The following heuristic
procedure guesses a synchronized
automaton for it.  Let $A = (\{0,1\} \times \{0,1\})^*$.
For two words $w,z \in \{0,1\}^*$ of equal length, we define
$w \times z$ to be the word in $A$
whose projection onto the first coordinate
is $w$ and second is $z$.

Define the language $L \subseteq A$ as follows:
$$ L = \{ w \times z \suchthat |w|=|z|,\ w,z \in \{0,1\}^*,\ \exists n\ 
	[w]_T = n,\ [z]_T = x_n \} .$$
Furthermore, define an equivalence relation on $A$ as follows:
$r \equiv_i s$ if and only if $rt \in L$ iff $st\in L$ for all
$t \in A$ with $|t| \leq i$.  Clearly this equivalence relation
partitions $A$ into finitely many equivalence classes.

Now do a breadth-first search on the elements of $A$ in increasing
order of length, where a word $r$ is marked as ``do not expand further''
if it is equivalent (under $\equiv_i$)
to a word that was previously examined.  In practice,
this is done by maintaining a queue of vertices to consider.
When the queue becomes empty, an automaton can be formed by letting the
states be the set of vertices, and a transition from $r$ to $s$
marked $a$ if $ra \equiv_i s$.  A state corresponding to $r$ is accepting
if $r \in L$.

One can then look at the resulting automata that are created for
increasing $i$.  If they stabilize (that is, the automata are the
same for two consecutive values of $i$), we can surmise that we have
found a synchronizing automaton.  The correctness of
the guessed automaton then needs to be verified rigorously
by some different approach.

\begin{remark}
In place of the language $L$ described above, one could also use
an alternative language $L'$, where one also insists
that neither the projections of $r$ nor $s$ contain three consecutive $1$'s.
It is not always clear which approach will result in smaller automata, and
either one can be used with {\tt Walnut}.  In fact, this is the one we 
actually used.
\end{remark}

When we carry out this procedure for the sequences $\bf X$ and
$\bf Y$ mentioned above, we easily find Tribonacci-synchronized
automata of $27$ and $30$ states, respectively.  In {\tt Walnut}
we call them {\tt xaut} and {\tt yaut}, and store them in the
{\tt Automaton Library} directory of {\tt Walnut}.
If the reader wishes to verify the computations in this paper,
the two automaton files can be downloaded from the author's
web page:\\
\centerline{\url{https://cs.uwaterloo.ca/~shallit/papers.html} \ .} \\
All the other {\tt Walnut} commands can simply be cut-and-pasted
into {\tt Walnut} in order to verify our claims.

\subsection{Verifying the automata}

Now that we have candidate automata, it remains to verify their correctness.
We can do this using mathematical induction.   We say that
a triple $(n,x,y)$ is {\it good\/} if all of the following conditions
hold:
\begin{enumerate}
\item $y > x$;
\item $x \not\in \{X(k) \suchthat 1 \leq k < n \}$
\item $y \not\in \{X(k) \suchthat 1 \leq k < n \}$
\item $x \not\in \{Y(k) \suchthat 1 \leq k < n \}$
\item $y \not\in \{Y(k) \suchthat 1 \leq k < n \}$
\item $y-x \not\in \{Y(k)-X(k) \suchthat 1 \leq k < n\}$
\item $y-x \not\in \{Y(k)+X(k) \suchthat 1 \leq k < n\}$
\end{enumerate}
To carry out the induction proof we must show three things:
\begin{enumerate}
\item The triple $(n, X(n), Y(n))$ is good for all $n\geq 1$;
\item If $(n, x,y)$ is good then $x \geq X(n)$;
\item If $(n,X(n),y)$ is good then $y \geq Y(n)$.
\end{enumerate}
The latter two conditions ensure that each value of $X(n)$ and $Y(n)$
chosen iteratively in Definition~\ref{def} is indeed the minimal
possible value among good candidates.

This verification can be carried out by the following {\tt Walnut} code.
To interpret it, you need to know the following basics of {\tt Walnut}
syntax:
\begin{itemize}
\item {\tt E} is the existential quantifier and {\tt A} is the universal
quantifier.
\item {\tt ?msd\_trib} tells {\tt Walnut} to represent integers in the
Tribonacci representation system.
\item {\tt \char'176} is logical NOT, {\tt \&} is logical AND, {\tt |} is logical
OR, {\tt =>} is logical implication.
\item {\tt def} defines an automaton accepting the values of the free
variables making the logical statement true.
\item {\tt eval} evaluates a logical statement with no free variables as
{\tt TRUE} or {\tt FALSE}.
\end{itemize}
Now we can create {\tt Walnut} predicates to verify the claims.
The code for {\tt good} asserts that its arguments $(n,x,y)$ represent
a good triple.   The three commands that follow it verify the three conditions.
\begin{verbatim}
def good "?msd_trib y>x &
   (~Ek k<n & $xaut(k,x)) &
   (~Ek k<n & $xaut(k,y)) &
   (~Ek k<n & $yaut(k,x)) &
   (~Ek k<n & $yaut(k,y)) &
   (~Ek,a,b k<n & $xaut(k,a) & $yaut(k,b) & y-x=b-a) &
   (~Ek,a,b k<n & $xaut(k,a) & $yaut(k,b) & y-x=b+a)":
eval check1 "?msd_trib An,x,y (n>=1 & $xaut(n,x) & $yaut(n,y)) => $good(n,x,y)":
eval check2 "?msd_trib An,x,y (n>=1 & $good(n,x,y)) => (Ez $xaut(n,z) & x>=z)":
eval check3 "?msd_trib An,x,y (n>=1 & $xaut(n,x) & $good(n,x,y)) =>
   (Ez $yaut(n,z) & y>=z)":
\end{verbatim}
and the last three commands all return {\tt TRUE}.

\section{Verifying a conjecture}

Now that we have rigorously-proven
automata for the sequences $\bf X$ and $\bf Y$, we
can use them to verify a conjecture of Robbert Fokkink,
made in a talk he gave on October 3 2022 \cite{Fokkink:2022}.
It also appears just before Section 2 in
the paper of Fokkink and Rust \cite{Fokkink&Rust:2022}.  (Notice
that they wrote $a_n$ for $X(n)$ and $b_n$ for $Y(n)$.)

\begin{conjecture}
For all $n \geq 1$ either $X(Y(n)) = X(n) + Y(n)$ 
or $X(Y(n)) = X(n) + Y(n) - 1$.
\end{conjecture}

\begin{proof}
We use the following {\tt Walnut} code:
\begin{verbatim}
eval fokkink "?msd_trib An,x,y,xy (n>=1 & $xaut(n,x) & $yaut(n,y)
   & $xaut(y,xy)) => (xy=x+y|xy=x+y-1)":
\end{verbatim}
and {\tt Walnut} returns {\tt TRUE}.
\end{proof}

We can also prove 
two conjectures of Julien Cassaigne, made
during the same talk:
\begin{conjecture}
\leavevmode
\begin{enumerate}
\item For all $n \geq 1$ either $X(Y(n))=X(n)+Y(n)$ or $Y(X(n))=X(n)+Y(n)$.
\item For all $n \geq 1$ we have $X(n) + Y(n) - Y(X(n)) \in \{0,1,2\}$.
\end{enumerate}
\end{conjecture}

\begin{proof}
We use {\tt Walnut} as follows:
\begin{verbatim}
eval julien1 "?msd_trib An,x,y,xy,yx ($xaut(n,x) & $yaut(n,y) & 
   $xaut(y,xy) & $yaut(x,yx)) => (xy=x+y|yx=x+y)":
eval julien2 "?msd_trib An,x,y,yx ($xaut(n,x) & $yaut(n,y) & $yaut(x,yx)) =>
   (x+y-yx=0|x+y-yx=1|x+y-yx=2)":
\end{verbatim}
and {\tt Walnut} returns {\tt TRUE} for both.
\end{proof}

Similarly, we can refute a conjecture of Dan Rust, made in
the same talk:
\begin{conjecture}
For all $n \geq 1$ we have
$X(Y(n)) + Y(X(n)) = 2X(n)+ 2Y(n) - 1$.
\end{conjecture}

\noindent{\it Disproof.}
We can {\it disprove} the conjecture as follows:
\begin{verbatim}
def rust "?msd_trib Ex,y,xy,yx $xaut(n,x) & $yaut(n,y) & $xaut(y,xy)
   & $yaut(x,yx) & xy+yx=2*x+2*y-1":
eval testrust1 "?msd_trib An (n>=1) => $rust(n)":
eval testrust2 "?msd_trib $rust(20)":
eval testrust3 "?msd_trib Am En (n>m & ~$rust(n))":
\end{verbatim}
Here {\tt Walnut} returns, respectively, {\tt FALSE}, {\tt FALSE}, and
{\tt TRUE}.  This shows that there are infinitely many counterexamples
to Rust's conjecture, and one of them is $n = 20$.   
\hfill $\qed$

\section{Other assorted properties}

Let's prove that $\bf X$ and $\bf Y$ are complementary sequences.
\begin{theorem}
The sets $\{ X(n) \suchthat n \geq 1 \}$ and
$\{ Y(n) \suchthat n \geq 1 \}$ form a disjoint partition of $\Enn
= \{ 1,2,3, \ldots \}$.
\end{theorem}

\begin{proof}
We use the following {\tt Walnut} code.
\begin{verbatim}
eval thm6 "?msd_trib (An (n>=1) => ((Ey $xaut(y,n))|(Ey $yaut(y,n))))
   & (An (n>=1) => (~Ey,z $xaut(y,n) & $yaut(z,n)))":
\end{verbatim}
and {\tt Walnut} returns {\tt TRUE}.
\end{proof}

Similarly, let's prove that $(Y(n)-X(n))_{n \geq 1}$ and
$(Y(n)+X(n))_{n \geq 1}$ are complementary sequences.
\begin{theorem}
The sets $\{ Y(n)-X(n) \suchthat n \geq 1 \}$ and
$\{ Y(n)+X(n) \suchthat n \geq 1 \}$ form a disjoint partition of $\Enn
= \{ 1,2,3, \ldots \}$.
\end{theorem}

\begin{proof}
We use the following {\tt Walnut} code:
\begin{verbatim}
def diff "?msd_trib Ex,y $xaut(n,x) & $yaut(n,y) & t=y-x":
def sum "?msd_trib Ex,y $xaut(n,x) & $yaut(n,y) & t=x+y":
eval thm7 "?msd_trib (An (n>=1) => ((Ey $diff(y,n))|(Ey $sum(y,n))))
   & (An (n>=1) => (~Ey,z $diff(y,n) & $sum(z,n)))":
\end{verbatim}
And {\tt Walnut} returns {\tt TRUE}.
\end{proof}

Define the sequence $a(n)$ (resp., $b(n)$, $c(n)$)
to be one more than the position of the $n$'th occurrence of the
symbol $0$ (resp., $1$, $2$) in $\bf TR$.
This is sequence \seqnum{A003144} (resp., \seqnum{A003145}, \seqnum{A003146})
in the OEIS.  Here the ``first occurrence''
means $n = 1$.   

\begin{theorem}
For $n \geq 1$ we have $Y(n) = a(n)+n$.
\end{theorem}

\begin{proof}
We use the following {\tt Walnut} code, where {\tt triba} is
a Tribonacci-synchronized automaton for $a(n)$.
\begin{verbatim}
reg shift {0,1} {0,1} "([0,0]|[0,1][1,1]*[1,0])*":
def triba "?msd_trib (s=0 & n=0) | Ex $shift(n-1,x) & s=x+1":
eval thm8 "?msd_trib An,x,y (n>=1 & $triba(n,x) & $yaut(n,y)) => y=x+n":
\end{verbatim}
And {\tt Walnut} returns {\tt TRUE}.  Here the definitions of
{\tt shift} and {\tt triba} come from \cite{Shallit:2022}.
\end{proof}

\begin{theorem}
For all $n \geq 0$ we have $X(n) = D_{n-1} + n$, where
$D_n$ is the sequence defined in Section~\ref{aut-sec}.
\end{theorem}

\begin{proof}
We use the following {\tt Walnut} code, where
{\tt tribd} is a Tribonacci-synchronized automaton for
$D_n$.
\begin{verbatim}
def tribd "?msd_trib Et,u $triba(s,t) & $triba(s+1,u) & t<=n & n<u":
eval thm9 "?msd_trib An,x,y (n>=1 & $tribd(n-1,x) & $xaut(n,y)) =>
   x+n=y":
\end{verbatim}
and {\tt Walnut} returns {\tt TRUE}.  Here the definition of
{\tt tribd} comes from \cite{Shallit:2022}.
\end{proof}

We can also find some relations between $\bf X$ and $\bf Y$ and
the sequences $a,b,c$ defined above.
\begin{theorem}
We have
\begin{itemize}
\item[(a)] $X(n) = b(n)-a(n)$ for $n \geq 0$;
\item[(b)] $Y(n) = c(n)-b(n)$ for $n \geq 0$.
\end{itemize}
\end{theorem}

\begin{proof}
We use the {\tt Walnut} commands
\begin{verbatim}
def tribb "?msd_trib (s=0&n=0) | Ex,y $shift(n-1,x) &
   $shift(x,y) & s=y+2":
def tribc "?msd_trib (s=0&n=0) | Ex,y,z $shift(n-1,x) &
   $shift(x,y) & $shift(y,z) & s=z+4":
eval parta "?msd_trib An,a,b,x ($triba(n,a) & $tribb(n,b) & $xaut(n,x)) 
   => x=b-a":
eval partb "?msd_trib An,b,c,y ($tribb(n,b) & $tribc(n,c) & $yaut(n,y)) 
   => y=c-b":
\end{verbatim}
and {\tt Walnut} returns {\tt TRUE} for both.
\end{proof}

Let $(\beta(n))_{n \geq 0}$ be the characteristic sequence of $\bf X$, that is
$\gamma(n) = 1$ if there exists $i$ such that $n= X(i)$, and
similarly let $(\gamma(n))_{n \geq 0}$ 
be the characteristic sequence of $\bf Y$.
Here $\beta$ is sequence \seqnum{A305385} and 
$\gamma$ is sequence \seqnum{A305386} in the OEIS.
\begin{theorem}
The sequences $\beta$ and $\gamma$ are Tribonacci-automatic.
\end{theorem}
\begin{proof}
We can easily
construct Tribonacci DFAO's generating these sequences, as follows:
\begin{verbatim}
def beta "?msd_trib Ei $xaut(i,n)":
combine BETA beta:
def gamma "?msd_trib Ei $yaut(i,n)":
combine GAMMA gamma:
\end{verbatim}
The automaton for $\beta$ is depicted in Figure~\ref{fig2}.
Here the notation $q/i$ in a state indicates that the state
is numbered $q$ and has output $i$.  To get the automaton for
$\gamma$, simply change the output of every state (except state $0$)
from $0$ to $1$ and vice versa. 
\end{proof}
\begin{figure}[H]
\begin{center}
\includegraphics[width=6.6in]{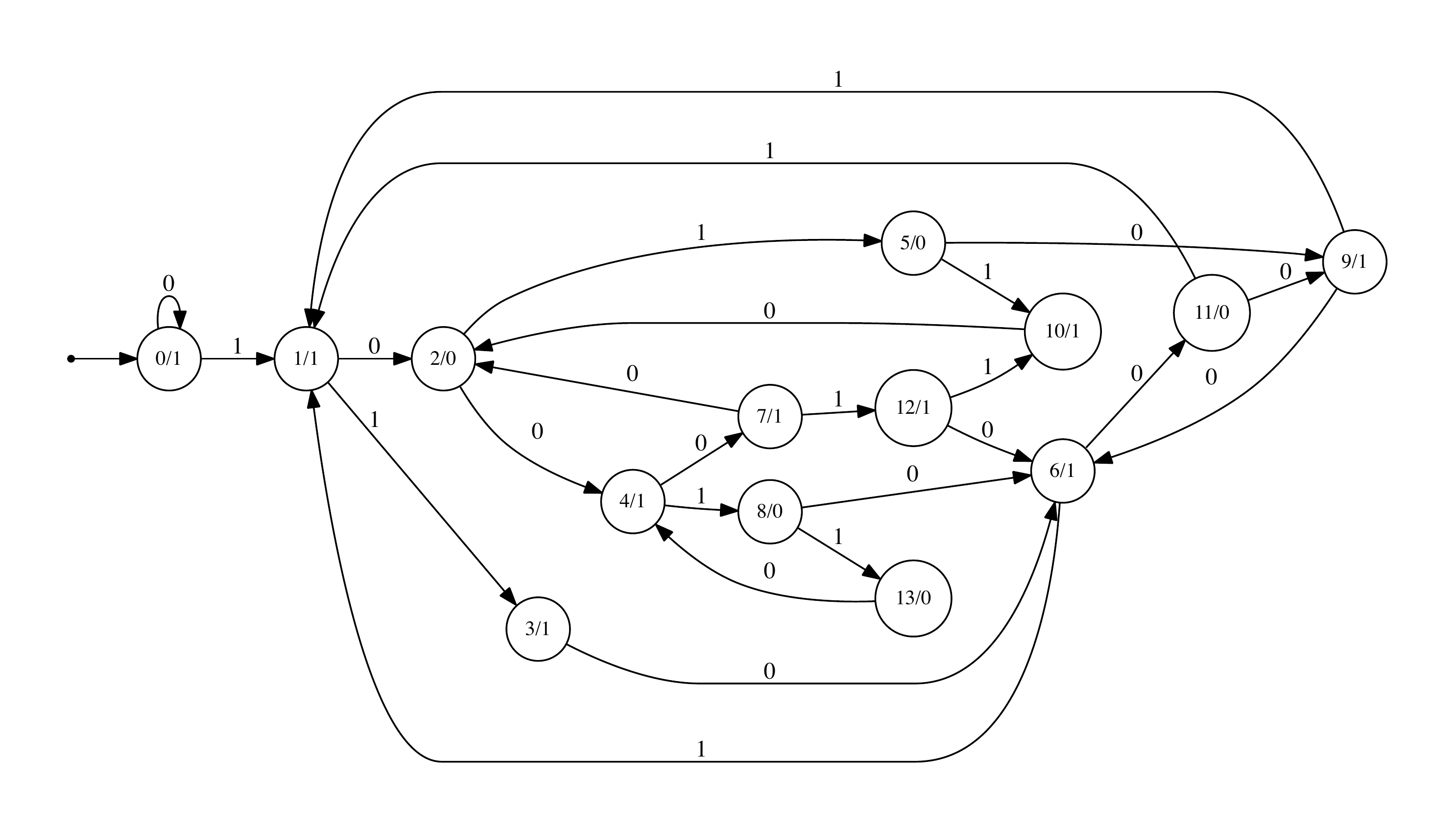}
\end{center}
\caption{Tribonacci DFAO computing $\beta$.}
\label{fig2}
\end{figure}

\section{Final remarks}

Nothing we have said in this note crucially depends on Tribonacci representation.
Exactly the same approach often works for sequences defined in
terms of Fibonacci representation, base-$k$
representation, and some other kinds of representations such as
Ostrowski representation.  See \cite{Shallit:2022} for a discussion
of what kinds of representations work.

{\tt Walnut} is available for free download at\\
\centerline{\url{https://cs.uwaterloo.ca/~shallit/walnut.html} \ .}\\

{\tt Walnut} has a number of limitations, and certainly cannot be used to
automatically prove all properties of all sequences, but in its small
domain it can be very effective, as we hope to have shown here.


\begin{thebibliography}{10}

\bibitem{Allouche&Shallit:2003}
J.-P. Allouche and J.~Shallit, {\em Automatic Sequences: Theory, Applications,
  Generalizations}, Cambridge University Press, 2003.

\bibitem{Barcucci&Belanger&Brlek:2004}
E.~Barcucci, L.~{B\'elanger}, and S.~Brlek, On {Tribonacci} sequences, {\em
  Fibonacci Quart.} {\bf 42} (2004), 314--319.

\bibitem{Carlitz&Scoville&Hoggatt:1972d}
L.~Carlitz, R.~Scoville, and V.~E. Hoggatt, Jr., Fibonacci representations of
  higher order, {\em Fibonacci Quart.} {\bf 10} (1972), 43--69,94.

\bibitem{Dekking&Shallit&Sloane:2020}
F.~M. Dekking, J.~Shallit, and N.~J.~A. Sloane, Queens in exile: non-attacking
  queens on infinite chess boards, {\em Electronic J. Combinatorics} {\bf 27}
  (2020), \#P1.52 (electronic).

\bibitem{Duchene&Rigo:2008}
E.~{Duch\^ene} and M.~Rigo, A morphic approach to combinatorial games: the
  {Tribonacci} case, {\em RAIRO Inform. Th\'eor. App.} {\bf 42} (2008),
  375--393.

\bibitem{Feinberg:1963}
M.~Feinberg, {Fibonacci--Tribonacci}, {\em Fibonacci Quart.} {\bf 1} (1963),
  71--74.

\bibitem{Fokkink:2022}
R.~Fokkink, A few words on games.
\newblock Talk for the One World Seminar on Combinatorics on Words, October 3
  2022. Available at
  \url{http://www.i2m.univ-amu.fr/wiki/Combinatorics-on-Words-seminar/_media/seminar2022:20221003fokking.pdf}.

\bibitem{Fokkink&Rust:2022}
R.~Fokkink and D.~Rust, Queen reflections---a modification of {Wythoff Nim}.
\newblock {\it Int. J. Game Theory}, 2022, to appear.

\bibitem{Huang&Wen:2017}
Y.~Huang and Z.~Wen, The numbers of repeated palindromes in the {Fibonacci} and
  {Tribonacci} words, {\em Disc. Appl. Math.} {\bf 230} (2017), 78--90.

\bibitem{Lejeune&Rigo&Rosenfeld:2020}
M.~Lejeune, M.~Rigo, and M.~Rosenfeld, Templates for the $k$-binomial
  complexity of the {Tribonacci} word, {\em Adv. in Appl. Math.} {\bf 112}
  (2020), 101947.

\bibitem{Mousavi:2016}
H.~Mousavi, Automatic theorem proving in {{\tt Walnut}}.
\newblock Arxiv preprint arXiv:1603.06017 [cs.FL], available at
  \url{http://arxiv.org/abs/1603.06017}, 2016.

\bibitem{Mousavi&Shallit:2015}
H.~Mousavi and J.~Shallit, Mechanical proofs of properties of the {Tribonacci}
  word.
\newblock In F.~Manea and D.~Nowotka, editors, {\em Proc. WORDS 2015}, Vol.
  9304 of {\em Lecture Notes in Computer Science}, pp.  1--21. Springer-Verlag,
  2015.

\bibitem{Richomme&Saari&Zamboni:2010}
G.~Richomme, K.~Saari, and L.~Q. Zamboni, Balance and {Abelian} complexity of
  the {Tribonacci} word, {\em Adv. in Appl. Math.} {\bf 45} (2010), 212--231.

\bibitem{Scott&Delaney&Hoggatt:1977}
A.~Scott, T.~Delaney, and V.~E. Hoggatt, Jr., The {Tribonacci} sequence, {\em
  Fibonacci Quart.} {\bf 15} (1977), 193--200.

\bibitem{Shallit:2021h}
J.~Shallit, Synchronized sequences.
\newblock In T.~Lecroq and S.~Puzynina, editors, {\em WORDS 2021}, Vol. 12847
  of {\em Lecture Notes in Computer Science}, pp.  1--19. Springer-Verlag,
  2021.

\bibitem{Shallit:2022}
J.~Shallit, {\em The Logical Approach to Automatic Sequences: Exploring
  Combinatorics on Words with {\tt Walnut}}, Cambridge University Press, 2022.
\newblock In press.

\bibitem{Sloane:2022}
N.~J.~A. Sloane et~al., The on-line encyclopedia of integer sequences, 2022.
\newblock Available at \url{https://oeis.org}.

\bibitem{Tan&Wen:2007}
B.~Tan and Z.-Y. Wen, Some properties of the {Tribonacci} sequence, {\em
  European J. Combinatorics} {\bf 28} (2007), 1703--1719.

\end{thebibliography}
\end{document}